\documentclass[11pt,reqno]{amsart}
\usepackage[utf8]{inputenc}
\usepackage[all]{xy} 
\usepackage{amsmath}
\usepackage{amssymb}
\usepackage{amsthm}
\usepackage{graphicx}
\usepackage{url}
\usepackage{tikz-cd}
\usepackage{mathtools}

\newtheorem{theorem}{Theorem}[section] 
\newtheorem{definition}[theorem] {Definition}
\newtheorem{lemma}[theorem]{Lemma}

\newtheorem{remark}[theorem]{Remark}
\newtheorem{proposition}[theorem]{Proposition}
\newtheorem{mainthm}{Theorem}

\newtheorem*{critpointspreserved}{Theorem~\ref{t:deckfixescrits}}

\theoremstyle{remark}
\newtheorem{example}[theorem]{Example}

\newcommand{\hatC}{\hat{\mathbb{C}}}

\DeclareMathOperator{\Deck}{Deck}

\DeclareMathOperator{\ran}{ran}
\newcommand{\Mob}{\textrm{M\"ob}}
 \newcommand{\Z}{\mathbb{Z}}
 \newcommand{\N}{\mathbb{N}}

\title{On the deck groups of iterates of bicritical rational maps}
\author{Sarah Koch}
\address{Department of Mathematics, University of Michigan, Ann Arbor, MI 48109,
U.S.A.}
\email{kochsc@umich.edu}
\author{Kathryn Lindsey}
\address{Department of Mathematics, Boston College, Chestnut Hill, MA 02467, U.S.A}
\email{kathryn.lindsey@bc.edu}
\author{Thomas Sharland}
\address{Department of Mathematics and Applied Mathematical Sciences, University of Rhode Island, RI 02881, U.S.A}
\email{tsharland@uri.edu}

\begin{document}

\begin{abstract} Given a rational map $f:\widehat{\mathbb C}\to\widehat{\mathbb C}$ on the Riemann sphere, we define $\Deck(f)$ to be the group of M\"obius transformations $\mu$ satisfying $f \circ \mu = f$. In this note, we consider the groups $\Deck(f^k)$, where $f$ is a \emph{bicritical} rational map (that is, a rational map with exactly two critical points) and $f^k$ denotes the $k$th iterate of $f$. In particular, we give a complete description of which groups (up to isomorphism) arise as the groups $\Deck(f^k)$ for bicritical rational maps $f$.
\end{abstract}

\maketitle

\section{Introduction}

Let $f:\hat{\mathbb{C}} \to \hat{\mathbb{C}}$ be a rational map on the Riemann sphere. We will denote the set of critical points (resp. critical values) of $f$ by $\mathcal{C}_f$ (resp. $\mathcal{V}_f$). A rational map $f$ is called bicritical if $|\mathcal{C}_f| = 2$. In addition, a bicritical rational map $f$ is called a power map if $\mathcal{C}_f = \mathcal{V}_f$. This is equivalent to the condition that $f$ is conjugate to $z \mapsto z^{\pm d}$ for some $d \geq 2$. 

In this note, we study particular groups of symmetries of the iterates of biciritical rational maps.  

\begin{definition}
The \emph{deck group} of a rational map $f:\hat{\mathbb{C}} \to \hat{\mathbb{C}}$ is the group 
\[
 \Deck(f) \coloneqq \{ \mu \in \Mob \mid f \circ \mu = f \},
\]
where $\Mob$ denotes the group of M\"obius transformations of the Riemann sphere $\hat{\mathbb{C}}$. 
\end{definition}

Deck groups were used in \cite{KLS1} as a tool for characterizing bicritical rational maps with shared iterates. Given a rational map $f$, the deck group $\Deck(f)$ is a finite subgroup of $\Mob$ (in fact, $|\Deck(f)| \leq \deg(f)$). It is well-known that a finite group of M\"obius transformations is isomorphic to either a cyclic group $\Z_n$ ($n \in \N$), a dihedral group $D_n$ ($n \in 2 \N$), or one of the \emph{polyhedral groups} $A_4$, $A_5$ or $S_4$ (see e.g \cite{Klein}).

The present work builds on the results of \cite{KLS1} to give a complete classification of which subgroups of the M\"obius group $\Mob$ are 
realized as deck groups of iterates of a bicritical rational map. Our main results are the following two theorems. 

\begin{mainthm}\label{t:odddegree}
 Let $f$ be a bicritical map of odd degree $d$. 
 \begin{enumerate}
\item Then $\Deck(f^k) \cong \mathbb{Z}_d$ for all $k \in \mathbb{N}$ if and only if $f$ is \emph{not} a power map. 
\item Furthermore, $f$ \emph{is} a power map if and only if there exists $k \in \mathbb{N}$ such that $\Deck(f^k) \cong \mathbb{Z}_n$ for some $n > d$.  
\end{enumerate}
\end{mainthm}

\begin{mainthm}\label{mthm} 
 Let $f$ be a bicritical map of even degree $d$. 
 \begin{enumerate}
 \item For each $k \in \mathbb{N}$, $\Deck(f^k)$ is isomorphic to either $D_{2d}$, $D_{4d}$, or $\mathbb{Z}_{d^n}$ for some $n \geq 1$.
\item Furthermore, if $f$ is not a power map then $|\Deck(f^k)| \leq 4d$.
 \end{enumerate}
\end{mainthm}

Additionally, in Proposition~\ref{p:examples} we give examples showing that when $d$ is even, each of the groups $D_{2d}$ and $D_{4d}$ are realized as $\Deck(f^k)$ for some degree $d$ bicritical rational map $f$ and some $1 \leq k \leq 3$, making the result of Theorem~\ref{mthm} sharp. A key step toward proving Theorems \ref{t:odddegree} and \ref{mthm} is the following result. 

\begin{theorem}\label{t:deckfixescrits}
 Let $f$ be a bicritical rational map and $\phi \in \Deck(f^k)$ for some $k$. Then $\phi(\mathcal{C}_f) = \mathcal{C}_f$.
 \end{theorem}

In \cite{Sym}, Pakovich studies the groups\footnote{Pakovich uses the notation $\Sigma(f^k)$ for $\Deck(f^k)$ and $\Sigma_\infty(f)$ for $\Deck_\infty(f)$.} $\Deck_\infty(f) \coloneqq \bigcup_{k=1}^\infty \Deck(f^k)$ for rational maps $f$. He shows that, if $f$ is not a power map, then $| \Deck_\infty(f) |$ is bounded, and this bound depends only on the degree $d$ of the map $f$. A study of rational maps of minimal degree with a given deck group\footnote{In \cite{HalfSymmetries}, elements of the deck group were called \emph{half-symmetries}.} was carried out in \cite{HalfSymmetries}. 

 This paper is structured as follows. In Section~\ref{s:background} we give the required background on deck groups and M\"obius transformations, as well as give sketch proofs of some important results from \cite{KLS1}. In Section~\ref{s:deckfixescrits}, we give a proof of Theorem~\ref{t:deckfixescrits}; the main difficulty is to prove the result in the quadratic case. After this, in Section~\ref{s:consequences} we will see that Theorem~\ref{t:deckfixescrits} implies a number of important results about deck groups of bicritical rational maps. In Section~\ref{s:specialmobius}, given a bicritical rational map $f$, we study the M\"obius transformations $\mu$ which satisfy $\mu(\mathcal{C}_f) = \mathcal{C}_f$ and $\mu(\mathcal{V}_f) = \mathcal{V}_f$. The results of Sections~\ref{s:consequences} and \ref{s:specialmobius} will then allow us to prove Theorems~\ref{t:odddegree} and \ref{mthm} in Section~\ref{s:proofs}.

\section{Background}\label{s:background}

For background on the dynamics of bicritical rational maps, we refer the reader to \cite{MilnorBicritical}. It can easily be verified that a bicritical rational map necessarily satisfies $|\mathcal{V}_f| =2$. By putting the critical points at $0$ and $\infty$, any bicritical rational map of degree $d$ is conjugate to a map of the form $z \mapsto \frac{\alpha z^d+\beta}{\gamma z^d+\delta}$. For such a map, the deck group is clearly generated by the map $z \mapsto e^{2 \pi i /d} z$. Thus if $f$ is bicritical of degree $d$, we have $\Deck(f) \cong \Z_d$.

Given a M\"obius transformation $\phi \colon \hatC \to \hatC$, denote by $\mathrm{Fix}(\phi)$ its set of fixed points. A non-identity M\"obius transformation  of finite order satisfies $|\mathrm{Fix}(\phi)| = 2$.   The following lemma collects together some standard properties of deck groups. We denote by $\deg_f(z)$ the \emph{local degree} of $f$ at the point $z \in \hatC$.

\begin{lemma}\label{l:standardfacts}
Let $f$ be a rational map of degree $d \geq 1$.
\begin{enumerate}
 \item The group $\Deck(f)$ is finite, and so must be cyclic, dihedral, or isomorphic to one of the polyhedral groups $A_4$, $A_5$ or $S_4$. Furthermore $|\Deck(f)| \leq d$.
 \item Any non-identity element of $\Deck(f)$ has exactly two fixed points.
 \item If $z \in \hatC$ and $\phi \in \Deck(f)$, then $\deg_f(z) = \deg_f(\phi(z))$.
 \item For all $k \geq 1$, $\Deck(f^{k}) \subseteq \Deck(f^{k+1})$.
 \item If $f$ is bicritical of degree $d$, then $\Deck(f) \cong \mathbb{Z}_d$. Furthermore, each non-identity $\phi \in \Deck(f)$ has $\mathrm{Fix}(\phi) = \mathcal{C}_f$.
\end{enumerate}

\end{lemma}

We will make use of the following classical characterization of commuting M\"obius transformations.

\begin{lemma}[\cite{Beardon:Groups}, Theorem 4.3.6] \label{l:Beardon}
 Let $\phi$ and $\mu$ be non-identity M\"obius transformations.
 Then the following are equivalent.
 \begin{enumerate}
  \item $\phi \circ \mu = \mu \circ \phi$
  \item $\phi(\mathrm{Fix}(\mu)) = \mathrm{Fix}(\mu)$ and $\mu(\mathrm{Fix}(\phi)) = \mathrm{Fix}(\phi)$.
  \item Either 
  \begin{enumerate}
\item   $\mathrm{Fix}(\mu) =\mathrm{Fix}(\phi)$, or 
\item $\phi$, $\mu$ and $\phi \circ \mu$ are involutions and $\mathrm{Fix}(\phi) \cap \mathrm{Fix}(\mu) = \varnothing$. 
\end{enumerate}
 \end{enumerate}
\end{lemma}

We now recall some results from \cite{KLS1} and sketch their proofs. The interested reader may refer to \cite{KLS1} for further details.

\begin{proposition}[\cite{KLS1}] \label{p:primeorders}
Let $f$ be a bicritical rational map of degree $d$, and let $p$ be a prime number that does not divide $d$.  Then for all natural numbers $k$, the group $\Deck(f^k)$ has no element of order $p$.
\end{proposition}

\begin{proof}[Sketch Proof]
 If some element $\tau \in \Deck(f^k)$ were to have order $p$, then each element in $\hatC$ would have orbit of length $1$ or $p$ under the action of $\langle \tau \rangle$. In particular the fiber $f^{-k}(w)$ over a regular value $w$ contains $d^k$ points, and since $p$ does not divide $d^k$, the element $\tau$ must fix at least one element in such a fiber. But then as there are infinitely many regular points for $f^k$, we see that $\tau$ must be the identity, a contradiction.
\end{proof}

\begin{theorem}[\cite{KLS1}] \label{t:cyclicOrDihedral}
Let $f$ be a bicritical rational map and $k \in \mathbb{N}$.  Then $\Deck(f^k)$ is either cyclic or dihedral. Furthermore, if the degree of $f$ is odd, then $\Deck(f^k)$ is cyclic.
\end{theorem}

\begin{proof}[Sketch Proof]
The polyhedral groups $A_4$, $A_5$ and $S_4$ all contain elements of order $2$ and elements of order $3$. Thus if $\Deck(f^k)$ were polyhedral we would have $\deg(f) = d \geq 6$ by Proposition~\ref{p:primeorders}. But then by Lemma~\ref{l:standardfacts}, then group $\Deck(f)$ would contain an element of order $d \geq 6$. But none of the polyhedral groups contain an element of order $\geq 6$, so this is a contradiction.

Now suppose the degree of $d$ is odd. In that case, $2$ does not divide $d$ and so by Proposition~\ref{p:primeorders}, $\Deck(f^k)$ cannot contain any element of order $2$. Thus $\Deck(f^k)$ is not dihedral.
\end{proof}

We remark that the following result, which is a special case of Theorem~\ref{mthm}, was obtained in \cite{KLS1}. The results in this paper are obtained by methods markedly different from those employed in \cite{KLS1}.

\begin{theorem}[\cite{KLS1}] \label{mthm3}
If $f$ is quadratic then the possibilities for $\Deck(f^k)$ (up to isomorphism) are $\mathbb{Z}_{2^n}$ for some $n \geq 1$, the Klein \emph{Vierergruppe} $V_4$ or the dihedral group $D_8$ of order $8$. Furthermore, if $f$ is not a power map then $|\Deck(f^k)| \leq 8$.
\end{theorem}

\section{Proof of Theorem \ref{t:deckfixescrits}}\label{s:deckfixescrits}

The primary goal of this section is to prove Theorem \ref{t:deckfixescrits}, which we reproduce here for convenience. 

\begin{critpointspreserved}
 Let $f$ be a bicritical rational map and $\phi \in \Deck(f^k)$ for some $k$. Then $\phi(\mathcal{C}_f) = \mathcal{C}_f$.
\end{critpointspreserved}
 
We begin by proving two useful lemmas.  First, Lemma \ref{l:Mobiuslemma} generalizes an argument from \cite{KLS1} (a similar result is also given by Pakovich in \cite{Sym}).

\begin{lemma}\label{l:Mobiuslemma}
Let $f$ be a degree $d$ bicritical rational map with critical point set $\mathcal{C}_f$ and critical value set $\mathcal{V}_f$. Then if $\phi$ is a M\"obius transformation such that $\phi(\mathcal{C}_f) = \mathcal{C}_f$, then there exists a unique M\"obius transformation $\mu$ such that $\mu \circ f = f \circ \phi$. Furthermore $\mu (\mathcal{V}_f) = \mathcal{V}_f$.
\end{lemma}

\begin{proof}
Once we prove existence, the uniqueness will follow from the surjectivity of $f$. We first prove the existence result for the sepcial case $g(z) = z^d$. In this case, $\phi$ is a M\"obius transformation such that $\phi(\mathcal{C}_g) = \mathcal{C}_g$ if and only if $\phi(z) = a z^{\pm 1}$ for some $a \in \mathbb{C} \setminus \{ 0 \}$. But then $g \circ \phi = a^d z^{\pm d}$, and so taking $\mu(z) = a^d z$ completes the proof for $g(z) = z^d$.

Now suppose that $f$ is an arbitrary bicritical rational map of degree $d$. Then there exist M\"obius transformations $\alpha$ and $\beta$ such that $f = \alpha \circ g \circ \beta$, where $g(z) = z^d$. In particular $\beta(\mathcal{C}_f) = \mathcal{C}_g$ and $\alpha(\mathcal{V}_g) = \mathcal{V}_f$. Thus if $\phi$ fixes $\mathcal{C}_f$ as a set then $\phi' = \beta \circ \phi \circ \beta^{-1}$ fixes $\mathcal{C}_g$ as a set, and by the above there exists $\mu'$ such that $\mu' \circ g = g \circ \phi'$. Hence taking $\mu = \alpha \circ \mu' \circ \alpha^{-1}$, a simple calculation yields 
\begingroup
\allowdisplaybreaks
\begin{align*}
\mu \circ f 	&= \mu \circ (\alpha \circ g \circ \beta) \\
			&= (\alpha \circ \mu' \circ \alpha^{-1}) \circ \alpha \circ g \circ \beta \\
			& = \alpha \circ (\mu' \circ g) \circ \beta \\
			&= \alpha \circ (f \circ \phi') \circ \beta \\
			&= \alpha \circ f \circ (\beta \circ \phi \circ \beta^{-1}) \circ \beta \\
			&= \alpha \circ g \circ \beta \circ \phi \\
			&= f \circ \phi
\end{align*}
\endgroup
as desired. The fact that $\mu (\mathcal{V}_F) = \mathcal{V}_F$ is clear.
\end{proof}

In particular, when the map $\phi$ in Lemma~\ref{l:Mobiuslemma} belongs to $\Deck(f^k)$, we get the following.

\begin{lemma}\label{l:keylemma}
Let $f$ be a bicritical rational map, $k  \geq 2$, and $\phi \in \Deck(f^k)$. If $\phi(\mathcal{C}_f) = \mathcal{C}_f$, then there exists a unique map $\phi_{k-1} \in \Deck(f^{k-1})$ such that $f \circ \phi_k = \phi_{k-1} \circ f$. Furthermore, $\phi_{k-1}(\mathcal{V}_f) = \mathcal{V}_f$.
\end{lemma}

\begin{proof}
By Lemma~\ref{l:Mobiuslemma}, it suffices to show that $\phi_{k-1} \in \Deck(f^{k-1})$. To see this, consider the following diagram.
 \begin{center}
  \begin{tikzcd} 
  \hat{\mathbb{C}} \arrow[r, "f"] \arrow[d, "\phi_k" ]
    &   \hat{\mathbb{C}} \arrow[d, "\phi_{k-1}" ]  \arrow[r, "f^{k-1}"]  &   \hatC  \arrow[d, "\mathrm{id}"]\\
      \hat{\mathbb{C}} \arrow[r, "f"]  &   \hat{\mathbb{C}}  \arrow[r, "f^{k-1}" ] &  \hatC 
\end{tikzcd}
\end{center}
The large outermost rectangle commutes since $\phi_k \in \Deck(f^k)$, and the square on the left comes directly from Lemma~\ref{l:Mobiuslemma}. Therefore, the square on the right commutes as well. As a consequence, $\phi_{k-1} \in  \Deck(f^{k-1})$.
\end{proof}


 
 Next, we enumerate some simple properties of dihedral groups that we will use later.   
 
 \begin{lemma}\label{l:dihedralfacts}
 Let $n \geq 2$. Consider the presentation of the dihedral group 
 \begin{equation}\label{e:dihedralpresentation}
D_{2n} =  \langle R, F \mid R^n = F^2 = (RF)^2 = \mathrm{id} \rangle.
\end{equation}
\begin{enumerate}
\item \label{i:dihedral} For each integer $c \geq 3$, $D_{2n}$ has at most one cyclic subgroup of order $c$.  Furthermore, any generator of such a cyclic group is a power of $R$.

\item \label{i:dihedralfact1} Let $n$ be even and suppose $\alpha \in D_{2n}$ has order $2$. Then there exists a subgroup $\Gamma$ of $D_{2n}$ such that $\alpha \in \Gamma$ and $\Gamma \cong V_4$. 
\end{enumerate}
\end{lemma}

\begin{remark} Since we allow $n=2$, we consider the \emph{Klein Vierergruppe} $V_4$ to be dihedral, i.e. $V_4 = D_{4}$. \end{remark} 

\begin{proof}
The proof of \eqref{i:dihedral} is a standard result about dihedral groups, and is left to the reader, who may wish to appeal to the characterization of $D_{2n}$ as the group of symmetries of the regular $n$-gon. For \eqref{i:dihedralfact1}, note first that the result trivially holds for the group $V_4$.  So assume $n>2$. Using the presentation \eqref{e:dihedralpresentation},
the center $Z(D_{2n})$ is isomorphic to $\mathbb{Z}_2$, and is generated by $\mu = R^{n/2}$. Thus if $\alpha$ is any other element of order $2$ in $D_{2n}$, then $\Gamma = \{ \mathrm{id}, \mu, \alpha, \mu \alpha = \alpha \mu \}$ forms a subgroup of $D_{2n}$ isomorphic to $V_4$.
\end{proof}
 
\subsection{Proof of Theorem~\ref{t:deckfixescrits} for degree $d \geq 3$}
 
 We now turn our attention to proving Theorem \ref{t:deckfixescrits}. The proof in the case for $\deg(f) \geq 3$ is relatively simple.
 
 \begin{lemma}\label{l:higherdegreefixescrits}
 Let $f$ be a bicritical rational map of degree $d \geq 3$ and suppose $\phi \in \Deck(f^k)$ for some $k$. Then $\phi(\mathcal{C}_f) = \mathcal{C}_f$.
 \end{lemma}

\begin{proof}
By Lemma~\ref{l:standardfacts}, $\Deck(f) \cong \Z_d$ and is generated by the $d$-fold rotation $\rho$ which fixes the points of  $\mathcal{C}_f$ pointwise. By Theorem \ref{t:cyclicOrDihedral},  $\Deck(f^k)$ is either cyclic or dihedral.

 If $\Deck(f^k)$ is cyclic, then all non-identity elements of $\Deck(f^k)$ have the same pair of fixed points. Since by Lemma~\ref{l:standardfacts} we have $\Deck(f^k) \subseteq \Deck(f)$, we see that $\rho \in \Deck(f^k)$. Thus for all non-identity $\phi \in \Deck(f^k)$, we have $\mathrm{Fix}(\phi) = \mathrm{Fix}(\rho) = \mathcal{C}_f$.
 
 If $\Deck(f^k)$ is dihedral, then Lemma \ref{l:dihedralfacts} part \eqref{i:dihedral} implies that $\Deck(f)$ is the unique cyclic subgroup of order $d$ in $\Deck(f^k)$, 
   and the generator $\rho$ of $\Deck(f)$ is an iterate of an element of maximal order in $\Deck(f^k)$. That is, using the presentation \eqref{e:dihedralpresentation}, $\rho = R^\ell$ for some $\ell \geq 1$. Hence all powers of $R$ fix the set $\mathrm{Fix}(\rho) = \mathcal{C}_f$. Since a generator $F$ as in \eqref{e:dihedralpresentation} is an involution, it must either swap the two critical points or fix them pointwise; either way $F$ fixes $\mathcal{C}_f$. Since the generators of $\Deck(f^k)$ fix $\mathcal{C}_f$ as a set, it follows that all elements of $\Deck(f^k)$ must fix $\mathcal{C}_f$ as a set.
\end{proof}
 
\subsection{Proof of Theorem~\ref{t:deckfixescrits} in the quadratic case} 
 
In the quadratic case a more careful analysis is required, since if $\Deck(f^k)$ is dihedral, it does not immediately follow that any order $2$ subgroup of $\Deck(f^k)$ must be the group $\Deck(f)$. To study this case, we introduce terminology from \cite{KLS1}.
 
\begin{definition}[\cite{KLS1}]\label{d:critcoal} We say a bicritical rational map with critical values $v_1$ and $v_2$ is \emph{critically coalescing} if $f(v_1) = f(v_2)$.
\end{definition}

We first give some simple properties of critically coalescing maps.

\begin{lemma}\label{l:critcoalmaps}
Let $f$ be a degree $d$ critically coalescing rational map with critical points $\mathcal{C}_f = \{c_1,c_2\}$ and critical values $\mathcal{V}_f = \{ v_1,v_2\}$, with $f(c_1) = v_i$ for $i=1,2$. Then
\begin{enumerate}
\item $\mathcal{C}_f \cap \mathcal{V}_f = \varnothing$.
\item $f^k(c_1) = f^k(c_2)$ for all $k \geq 2$
\end{enumerate}
\end{lemma}

\begin{proof}\mbox{}
\begin{enumerate}
\item Without loss of generality suppose $v_1 \in \mathcal{C}_f \cap \mathcal{V}_f$. Then $\{ v_1, f(v_2) \} \subseteq f^{-1}(f(v_1))$, so $f(v_1)$ has $d+1$ preimages, counting multiplicity. This is a contradiction.
\item  This follows easily from the part (i) and the fact that $f(v_1) = f(v_2)$.
\end{enumerate}
\end{proof}

 
 The next lemma is a slight generalization of a result from \cite{KLS1}. It relates Definition~\ref{d:critcoal} with the groups $\Deck(f^k)$ for a bicritical rational map.

\begin{lemma}\label{l:dihedralimpliescritcoal}
Let $f$ be a bicritical rational map of even degree $d$. If $\Deck(f^k)$ is dihedral for some $k \in \mathbb{N}$, then $f$ is critically coalescing.
\end{lemma}

\begin{proof}
Write $\Deck(f) = \langle \tau \rangle$ and suppose $\mu = \tau^{d/2}$ is the unique element of order $2$ in $\Deck(f)$. Suppose $k>1$ is minimal such that $\Deck(f^k)$ is dihedral. By Lemma~\ref{l:dihedralfacts}, there exists $\Gamma$, a subgroup of $\Deck(f^k)$ such that $\Gamma \cong V_4$ and $\mu \in \Gamma$. Write $\Gamma = \{ \mathrm{id}, \mu, \alpha, \beta \}$, where $\alpha$ and $\beta$ are order $2$ elements of $\Deck(f^k)$. Since $\Gamma$ is abelian and $\mathrm{Fix}(\mu) = \mathcal{C}_f$, then by Lemma~\ref{l:Beardon} we have $\alpha(\mathcal{C}_f) = \beta(\mathcal{C}_f) = \mathcal{C}_f$. Thus, by Lemma~\ref{l:keylemma}, there exists $\nu \in \Deck(f^{k-1})$ such that $\nu \circ f = f \circ \alpha$. Furthermore,  $\nu(\mathcal{V}_f ) = \mathcal{V}_f = \{v_1, v_2 \}$. Since $\alpha$ is an order $2$ element of $\Deck(f^k)$ distinct from $\mu$, it cannot be an element of $\Deck(f)$. Hence $f \circ \alpha \neq f$, and so $\nu \neq \mathrm{id}$. By the assumption on the minimality of $k$, $\Deck(f^{k-1})$ must be cyclic. Since all non-identity elements of a finite cyclic group of M\"obius transformations share the same pair of fixed points, we see that for all non-identity elements $\gamma \in \Deck(f^{k-1})$ we have $\mathrm{Fix}(\gamma) = \mathrm{Fix}(\mu) = \mathcal{C}_f$. In particular $\mathrm{Fix}(\nu) = \mathcal{C}_f$.

Suppose that $\nu$ fixes the elements of $\mathcal{V}_f$ pointwise. Then we have $\mathcal{C}_f = \mathcal{V}_f$, and so $f$ is a power map. But this is impossible, since $\Deck(f^k)$ is always cyclic for power maps. So $\nu$ must swap the elements of $\mathcal{V}_f$, and so $\nu$ is an involution. But since $\Deck(f^{k-1})$ is cyclic, it contains at most one involution. But since $\mu$ is an order $2$ element of $\Deck(f) \subseteq \Deck(f^{k-1})$, me must have $\nu = \mu$. Thus $\nu = \mu \in \Deck(f)$ interchanges the elements of $\mathcal{V}_f$, and so $f(v_1) = f(v_2)$.
\end{proof}

It was shown in \cite{KLS1} that the converse to the above lemma is true in the quadratic case. However, in Example~\ref{ex:critcoalconversefalse} we will show that for higher degrees, the converse to Lemma \ref{l:dihedralimpliescritcoal} does not hold. 


The following is a simple topological observation.

\begin{lemma}\label{l:localdegree}
 Let $F$ be a rational map of degree $d$ and suppose $\Deck(F)$ contains an element of order $k$. Then there exists $z \in \hatC$ such that $\deg_{F}(z) \geq k$. 
\end{lemma}

\begin{proof}
 First notice that if $\phi$ has order $k$, then for every $\zeta \notin \mathrm{Fix}(\phi)$, the points  $\zeta, \phi(\zeta), \dotsc , \phi^{k-1}(\zeta)$ must all be distinct. Otherwise, if $\phi^{i}(\zeta) = \phi^{j}(\zeta)$ for some $0 \leq i < j \leq k-1$, then $\phi^{j-i}$ fixes the point $\phi^{i}(\zeta)$. But then $\mathrm{Fix}(\phi^{j-i}) \supseteq \mathrm{Fix}(\phi) \cup \{ \phi^{i} \}$, so that $|\mathrm{Fix}(\phi^{j-1})| \geq 3$. This implies $\phi^{j-i} = \mathrm{id}$. Since $j - i < k$, this is a contradiction.
 
 Let $z \in \mathrm{Fix}(\phi)$. Suppose that $V$ is a simply connected neighborhood of $F(z)$ such that $V \cap \mathcal{V}_F \subseteq \{ F(z) \}$, and let $U$ be the component of $F^{-1}(V)$ which contains $z$. By restricting $V$ if necessary, we may assume that $F^{-1}(F(z)) \cap U = \{z\}$ and that $z$ is the only element of $\mathrm{Fix}(\phi)$ in $U$. 
 
 Let $w \neq F(z)$ be an element of $V$.  Then there exists $u_0 \in U$ such that $F(u_0) = w$. For each $0 \leq j \leq k-1$, define $u_j = \phi^j(u_0)$. From the above, we know that $u_i \neq u_j$ for $i \neq j$. Furthermore, since $F = F \circ \phi^j$ and $\phi^j(z) = z$ for all $0 \leq j \leq k-1$, we see that $u_j \in U$ for all $0 \leq j \leq k-1$. Hence $F \colon U \to V$ is a branched covering of degree (at least) $k$ on the simply connected set $U$. Since the only critical point in $U$ is $z$, we have $\deg_{F}(z) \geq k$ by the Riemann-Hurwitz Theorem.
\end{proof}

We are particularly interested in applying the previous result to the case where $F = f^k$ is an iterate of a quadratic rational map $f$ and $\Deck(f^k)$ is dihedral.

\begin{lemma}\label{l:periodiccritpoint}
Let $f$ be a quadratic rational map and suppose $\Deck(f^k)$ is dihedral. If $f^k$ has a critical point with local degree greater than $2$, then one of the critical points $c_1$ of $f$ is periodic of some period $p$. Furthermore:
\begin{itemize}
 \item the second critical point $c_2$ satisfies $f^p(c_2) = c_1$ and
 \item for either critical point $c$, $f^n(c) = c_1$ if and only if $n = ap$ for some $a \geq 0$.
\end{itemize}

\end{lemma}

\begin{proof}
By Lemma~\ref{l:dihedralimpliescritcoal}, $f$ must be critically coalescing. Let $\mathcal{C}_f = \{ c_1, c_2\}$.  Since $f$ is critically coalescing, we see that by Lemma~\ref{l:critcoalmaps} that $f(c_i) \notin \mathcal{C}_f$ for $i=1,2$. 

A point $z \in \hatC$ maps forward with local degree greater than $1$ under $f^k$ if and only if  $z$ is a preimage $f^{-j}(c_i)$ for some $0 \leq j < k$ and $i =1,2$. Furthermore, if $z$ maps forward by local degree strictly greater than $2$ under $f^k$, then the forward orbit
\[
 \mathcal{O}_k(z) = ( z, f(z), f^2(z), \dotsc, f^{k-1}(z) )
\]
must contain (at least) two critical points of $f$. If the same critical point $c_i$ appears twice, we are done, since then that critical point would be periodic. So assume without loss of generality that there exist $0 \leq n < m < k$ with $f^n(z) = c_2$ and $f^m(z) = c_1$. Then we have $f^{m-n}(c_2) = c_1$. But since $f$ is critically coalescing, then by Lemma~\ref{l:critcoalmaps} we have $f^{\ell}(c_1) = f^{\ell}(c_2)$ for all $\ell \geq 2$. Thus $f^{m-n}(c_1) = c_1$ and so $c_1$ is a periodic critical point under $f$.

Now suppose that $p$ is the period of $c_1$, so that $p>0$ is minimal such that $f^p(c_1) = c_1$. Since $f$ is critically coalescing, we also have $f^p(c_2) = c_1$, and if there were $0<j<p$ such that $f^j(c_2) = c_1$, then this would imply $f^j(c_1) = c_1$, which is a contradiction.
\end{proof}

\begin{proposition}\label{p:quadraticdeckfixescrits}
Let $f$ be a quadratic rational map and suppose that for some $k \in \mathbb{N}$ the group $\Deck(f^k)$ is dihedral. If $\phi \in \Deck(f^k)$ then $\phi(\mathcal{C}_f) = \mathcal{C}_f$.
\end{proposition}

\begin{proof}
If $\Deck(f^k)$ is isomorphic to $V_4$, then $\Deck(f^k)$ is abelian. Thus every element of $\Deck(f^k)$ commutes with $\mu$, the unique order $2$ element of $\Deck(f)$. Hence by Lemma~\ref{l:Beardon}, since $\mathrm{Fix}(\mu) = \mathcal{C}_f$, we have $\phi(\mathcal{C}_f) = \mathcal{C}_f$ for all $\phi \in \Deck(f^k)$.

We now assume that $\Deck(f^k)$ is dihedral and contains an element of order $n>2$. By Lemma~\ref{l:localdegree}, there must exist $z \in \hatC$ such that $\deg_{f^k}(z) \geq n > 2$ and so by Lemma~\ref{l:periodiccritpoint}, $f$ has a periodic critical point, $c_1$, and the other critical point $c_2$ eventually maps onto $c_1$, but is not in the forward orbit of $c_1$. We will show that the orbit $\mathrm{orb}_{\Deck(f^k)}(c_1)$ under the action of $\Deck(f^k)$ is equal to $\mathcal{C}_f = \{ c_1, c_2\}$. To see that $c_2 \in \mathrm{orb}_{\Deck(f^k)}(c_1)$, let $\Gamma$ be a subgroup of $\Deck(f^k)$ which is isomorphic to $V_4$ and which contains $\mu$ (such a subgroup exists by Lemma~\ref{l:dihedralfacts}). By Lemma~\ref{l:Beardon}, every non-identity element $\phi \in \Gamma$ must have $\phi(\mathcal{C}_f) = \mathcal{C}_f$. In particular, if $\phi \neq \mu$, then since $\mu$ and $\phi$ are distinct and both have order $2$, we see that $\phi$ must transpose the elements of $\mathrm{Fix}(\mu) = \mathcal{C}_f$. Hence $\phi(c_1) = c_2$, meaning $\{c_1,c_2\} = \mathcal{C}_f \subseteq \mathrm{orb}_{\Deck(f^k)}(c_1)$.

To prove the reverse inclusion, suppose that $a \in \mathrm{orb}_{\Deck(f^k)}(c_1)$, so that there exists $\phi \in \Deck(f^k)$ such that $\phi(c_1) = a$. Let $p \geq 2$ be the period of $c_1$. By Lemma~\ref{l:standardfacts}, we have $\deg_{f^s}(a) = \deg(f^s)(c_1)$ for all $s \geq k$. Thus, for all $j \geq 0$ and $1 \leq m \leq p$ such that $jp +m \geq k$ we have 
\begin{equation}\label{e:localdegree}
\deg_{f^{jp + m}}(a) = \deg_{f^{jp + m}}(c_1) = 2^{j+1}.
\end{equation}
Since $\deg_{f^{jp + m}}(a)>1$, it follows that $a$ must eventually map onto a critical point. Let $q \geq 0$ be minimal such that $f^q(a) \in \mathcal{C}_f$.

We now show that $a \in \mathcal{C}_f$. Let $j$ be minimal such that $jp + 1 \geq k$. Then by \eqref{e:localdegree}, we have $\deg_{f^{jp + 1}}(a) = 2^{j+1}$ and so the orbit
\[
 \mathcal{O}_{jp}(a) = ( a, f(a), f^2(a), \dotsc, f^{jp}(a) )
\]
must contain exactly $j+1$ critical points. Suppose $a \notin \mathcal{C}_f$ so that $q >0$. By Lemma~\ref{l:periodiccritpoint}, the points $f^q(a),f^{p+q}(a),f^{2p+q}(a),\dotsc $ are all critical. But since $q>0$, the inequality $ip + q \leq jp$ has at most $j$ solutions for $i \geq 0$. Thus there are at most $j$ critical points in the orbit $\mathcal{O}_{jp}(a)$, which is a contradiction. Thus $q=0$ and so $a \in \mathcal{C}_f$. Hence $\mathrm{orb}_{\Deck(f^k)}(c_1) \subseteq \mathcal{C}_f$. We conclude that $\mathrm{orb}_{\Deck(f^k)}(c_1) = \mathrm{orb}_{\Deck(f^k)}(c_1) = \mathcal{C}_f$. It follows that $\phi(\mathcal{C}_f) = \mathcal{C}_f$ for all $\phi \in \Deck(f^k)$.
\end{proof}

 \begin{proof}[Proof of Theorem~\ref{t:deckfixescrits}]
 In light of Lemma~~\ref{l:higherdegreefixescrits}, we only need to prove the result in the quadratic case. By Theorem~\ref{t:cyclicOrDihedral}, $\Deck(f^k)$ is either cyclic or dihedral. If $\Deck(f^k)$ is cyclic, then every non-identity element of $\Deck(f^k)$ has the same set of fixed points. But since the unique order $2$ element $\mu \in \Deck(f)$ has $\mathrm{Fix}(\mu) = \mathcal{C}_f$. Since $\mu \in \Deck(f^k)$, we see that for all non-identity $\phi \in \Deck(f^k)$ we have $\mathrm{Fix}(\phi) = \mathcal{C}_f$. Hence in this case $\phi(\mathcal{C}_f) = \mathcal{C}_f$. On the other hand, if $\Deck(f^k)$ is dihedral then the result holds by Proposition~\ref{p:quadraticdeckfixescrits}.
 \end{proof}
 
 \section{Consequences of Theorem~\ref{t:deckfixescrits}}\label{s:consequences}
 
 Theorem~\ref{t:deckfixescrits} has a number of useful consequences. We first state a strengthened version of Lemma~\ref{l:keylemma}.

\begin{proposition}\label{p:keyprop}
 Let $f$ be a bicritical rational map and $\phi_k \in \Deck(f^k)$ for some $k$. Then there exists a unique $\phi_{k-1} \in \Deck(f^{k-1})$ such that $f \circ \phi_k = \phi_{k-1} \circ f$. Moreover $\phi_{k-1}(\mathcal{V}_f) = \mathcal{V}_f$.
\end{proposition}

\begin{proof}
 The proof is the same as Lemma~\ref{l:keylemma}, with the hypothesis that $\phi_k(\mathcal{C}_f) = \mathcal{C}_f$ removed by Theorem~\ref{t:deckfixescrits}.
\end{proof}

We now use Proposition~\ref{p:keyprop} to prove a number of preliminary results which we will use to prove the main theorems. Observe that by Proposition~\ref{p:keyprop}, if for some $k > 1$ we have $\phi_k \in \Deck(f^k)$, then we can recursively define a sequence
\[
 (\phi_k , \phi_{k-1}, \dotsc, \phi_1, \phi_0 = \mathrm{id} )
\]
where for each $j$,  $\phi_j \in \Deck(f^j)$ and $f \circ \phi_j = \phi_{j-1} \circ f$. Each $\phi_j$ is uniquely determined by the initial choice of $\phi_{k}$, and we must have $f^{k-j} \circ \phi_{k} = \phi_j \circ f^{k-j}$. This gives the following commutative diagram.

\begin{center}
  \begin{tikzcd} 
        \hatC \arrow[r, "f"]  \arrow[d, "\phi_k" ]
    &   \hatC \arrow[r, "f" ] \arrow[d, "\phi_{k-1}" ]
    &   \dots \arrow[r, "f" ] 
    &   \hatC \arrow[r, "f" ] \arrow[d, "\phi_1" ]
    &   \hatC                 \arrow[d, "\mathrm{id}" ] \\
        \hatC \arrow[r, "f"] 
    &   \hatC \arrow[r, "f" ] 
    &   \dots \arrow[r, "f" ] 
    &   \hatC \arrow[r, "f" ] 
    &   \hatC                 
\end{tikzcd}
\end{center}

In particular, if $j=1$ we obtain the following result.

\begin{proposition}\label{p:Sarah1}
Let $f$ be a bicritical rational map. Let $k > 1$ and suppose $\phi_k \in \Deck(f^k)$. Then there exists a unique $\phi_1 \in \Deck(f)$ such that the following diagram commutes.
\begin{center}
  \begin{tikzcd} 
        \hatC \arrow[r, "f^{k-1}"] \arrow[d, "\phi_k" ']
    &   \hatC \arrow[r, "f" ] \arrow[d, "\phi_1"]
    &   \hatC \arrow[d, "\mathrm{id}" ] \\
        \hatC \arrow[r, "f^{k-1}" ]
    &   \hatC \arrow[r, "f" ]
    &   \hatC  
\end{tikzcd}
\end{center}
Furthermore
\begin{enumerate}
 \item $\phi_1$ is the identity if and only if $\phi_k \in \Deck(f^{k-1})$.
 \item $\phi_1(\mathcal{V}_f) = \mathcal{V}_f$.
\end{enumerate}
\end{proposition}

\begin{proof}
By Proposition~\ref{p:keyprop} and the discussion following Proposition~\ref{p:keyprop}, we know that there exists a unique $\phi_1 \in \Deck(f)$ such that $f^{k-1} \circ \phi_k = \phi_1 \circ f^{k-1}$. This proves the diagram commutes. Now suppose that $\phi_k \in \Deck(f^{k-1})$. Then since $f^{k-1} = f^{k-1} \circ \phi_k$, we see from the diagram that $\phi_1 = \mathrm{id}$. On the other hand, if $\phi_1 = \mathrm{id}$, then the diagram shows that $f^{k-1} \circ \phi_k = f^{k-1}$, and so $\phi_k \in \Deck(f^{k-1})$. The assertion that $\phi_1(\mathcal{V}_f) = \mathcal{V}_f$ again follows from Proposition~\ref{p:keyprop}.
\end{proof}

It should be noted that in general, an element $\phi \in \Deck(f)$ need not map $\mathcal{V}_f$ to itself. For example, if $f(z) = \frac{1}{z^2-1}$, then the unique non-identity element of $\Deck(f)$ is $\phi(z) = -z$, which fixes the critical points $0$ and $\infty$ of $f$. However, $\mathcal{V}_f = \{ 0 , -1 \}$, which is clearly not preserved by $\phi$. 

The following can be thought of as a partial converse to Proposition~\ref{p:keyprop}.

\begin{lemma}\label{l:keypropconverse}
Let $f$ be a bicritical rational map of degree $d$ and suppose $\mu$ is a M\"obius transformation such that $\mu(\mathcal{V}_f) = \mathcal{V}_f$. Then there exists a M\"obius transformation $\phi$ such that $f \circ \phi = \mu \circ f$ and $\phi(\mathcal{C}_f) = \mathcal{C}_f$. In particular, if $\mu \in \Deck(f^{k})$ for some $k \geq 1$ then $\phi \in \Deck(f^{k+1})$.
\end{lemma}

\begin{proof}
The proof proceeds like that of Lemma~\ref{l:Mobiuslemma}, but note in this case there is no uniqueness. First, suppose $g(z) = z^d$. Then if $\mu(\mathcal{V}_g) = \mathcal{V}_g = \{0,\infty\}$, we have $\mu(z) = az^{\pm 1}$ for some $a \in \mathbb{C} - \{0\}$. Thus taking $\phi(z) = a^d z^{\pm 1}$, we get $g \circ \phi = \mu \circ g$ and $\phi(\mathcal{C}_g) = \mathcal{C}_g$ as required.

For the general case, we again note that if $f$ is a bicritical rational map of degree $d$, then there exist M\"obius transformations $\alpha$ and $\beta$ such that $f = \alpha \circ g \circ \beta$ for $g(z) = z^d$, so that $\alpha(\mathcal{V}_g) = \mathcal{V}_f$ and $\beta(\mathcal{C}_f) = \mathcal{C}_g$. Thus if $\mu$ is a M\"obius transformation and $\mu(\mathcal{V}_f) = \mathcal{V}_f$, then $\mu' = \alpha^{-1} \circ \mu \circ \alpha$ must satisfy $\mu'(\mathcal{V}_g) = \mathcal{V}_g$. Hence, by the previous paragraph, there exists $\phi'$ such that $g \circ \phi' = \mu' \circ g$ and $\phi'(\mathcal{C}_g) = \mathcal{C}_g$. Thus taking $\phi = \beta^{-1} \circ \phi' \circ \beta$ we get $f \circ \phi = \mu \circ f$ and $\phi(\mathcal{C}_f) = \mathcal{C}_f$, as required.

Finally, if $\mu \in \Deck(f^k)$ then since $f \circ \phi = \mu \circ f$, composing on the left by $f^k$ gives
\[
f^{k+1} \circ \phi = f^k \circ (f \circ \phi) = f^k \circ (\mu \circ f) = (f^k \circ \mu) \circ f   = f^{k} \circ f = f^{k+1}
\]
and so $\phi \in \Deck(f^{k+1})$.
\end{proof}

\begin{lemma}\label{l:eventuallyconstant}
Let $f$ be a bicritical rational map. Suppose for some $k$ that $\Deck(f^k) = \Deck(f^{k+1})$. Then $\Deck(f^{k+2}) = \Deck(f^{k+1}) = \Deck(f^k)$.
\end{lemma}

\begin{proof}
Let $\phi \in \Deck(f^{k+2})$. By Proposition~\ref{p:keyprop}, there exists $\mu \in \Deck(f^{k+1})$ such that
\begin{equation}\label{e:eventuallyconstant}
 f \circ \phi = \mu \circ f.
\end{equation} 
Since $\Deck(f^{k+1}) = \Deck(f^k)$, we see that $\mu \in \Deck(f^k)$. But then postcomposing \eqref{e:eventuallyconstant} by $f^k$ yields
\[
 f^{k+1} \circ \phi = f^k \circ (\mu \circ f) = (f^k \circ \mu) \circ f = f^{k+1}
\]
and so $\phi \in \Deck(f^{k+1})$.
\end{proof}

\begin{remark}
The authors are unaware if there are counterexamples to Lemma~\ref{l:eventuallyconstant} in the case where $f$ is a general rational map. That is, if $f$ is a rational map and $k \geq 1$, is it true that $\Deck(f^k) = \Deck(f^{k+1})$ implies $\Deck(f^n) = \Deck(f^k)$ for all $n \geq k$?
\end{remark}

In the following, we will use the notation $\Deck^\ast(f^k)=\Deck(f^k) \setminus \Deck(f^{k-1})$, with the convention that $\Deck(f^0) = \Deck^\ast(f^0) = \{ \mathrm{id} \}$. Note that by using this convention, we have $\Deck^\ast(f) = \Deck(f) \setminus \{ \mathrm{id} \}$ is the set of non-identity elements of $\Deck(f)$.

\begin{lemma}\label{l:deckincreases}
 Let $f$ be a bicritical rational map and $k \geq 0$. Then $\Deck^\ast(f^{k+1}) \neq \varnothing$ if and only if there exists $\mu \in \Deck^\ast(f^k)$ such that $\mu(\mathcal{V}_f) = \mathcal{V}_f$.
\end{lemma}

\begin{proof}
 Suppose $\phi \in \Deck^\ast(f^{k+1})$. By Proposition~\ref{p:keyprop}, there exists $\mu \in \Deck(f^k)$ such that 
 \begin{equation}\label{eq:Symmetryidentity}
 \mu \circ f = f \circ \phi 
 \end{equation}
 and $\mu(\mathcal{V}_f) = \mathcal{V}_f$. If $\mu$ is not an element of $\Deck^\ast(f^k)$, then $\mu \in \Deck(f^{k-1})$. Thus $f^{k-1} \circ \mu = f^{k-1}$, and so composing $f^{k-1}$ on the left of \eqref{eq:Symmetryidentity} we get
 \[
  f^{k} = f^{k-1} \circ \mu \circ f = f^{k} \circ \phi
 \]
and so $\phi \in \Deck(f^k)$. But this contradicts $\phi \in \Deck^\ast(f^{k+1})$, and so we conclude that $\mu \in \Deck^\ast(f^k)$.
 
 Conversely, suppose that there exists $\mu \in \Deck^\ast(f^k)$ such that $\mu(\mathcal{V}_f) = \mathcal{V}_f$. It follows from Lemma~\ref{l:keypropconverse} that there exists $\phi \in \Deck(f^{k+1})$ such that $\mu \circ f = f \circ \phi$. Suppose that $\phi \in \Deck(f^k)$. Then Proposition~\ref{p:keyprop} asserts that there is a unique $\mu' \in \Deck(f^{k-1})$ such that $\mu' \circ f = f \circ \phi$. But then $\mu = \mu'$, and so this contradicts $\mu \in \Deck^\ast(f^k)$. Hence $\phi \in \Deck^\ast(f^{k+1})$ and so $\Deck^\ast(f^{k+1}) \neq \varnothing$.  
\end{proof}

Before continuing, we provide an example which shows that the converse to Lemma~\ref{l:dihedralimpliescritcoal} does not generalize to higher degrees.

\begin{example}\label{ex:critcoalconversefalse}
 Let $f(z) = \frac{z^4-1}{z^4+i}$. Then $\mathcal{V}_f = \{1, i\}$ and $f(1) = f(i) = 0$; thus $f$ is critically coalescing. One can check by direct computation that $\Deck(f^2) = \Deck(f) \cong \Z_4$ and then appeal to Lemma~\ref{l:eventuallyconstant}, but here is an argument that also makes use of Lemma~\ref{l:deckincreases}. It is easy to see that $\Deck(f)$ is generated by the order $4$ rotation $\rho(z) = i z$. But then for all $n \in \{1,2,3\}$, we have $\rho^n(\mathcal{V}_f) \neq \mathcal{V}_f$, and so there does not exist a non-identity element of $\Deck(f)$ which fixes $\mathcal{V}_f$ as a set. From Lemma~\ref{l:deckincreases}, we see that $\Deck(f^2) = \Deck(f) \cong \Z_4$. Then by  Lemma~\ref{l:eventuallyconstant}, we see that $\Deck(f^k) \cong \Z_4$ for all $k \geq 1$.
\end{example}

\begin{lemma}\label{l:decksizes}
 Let $f$ be a bicritical rational map of degree $d \geq 2$. Then
 \begin{equation}\label{eq:decksizes}
  \frac{|\Deck(f^k)|}{|\Deck(f^{k-1})|} \leq d.
 \end{equation}
Furthermore, if $f$ is not a power map, then the quotient is at most $2$.
\end{lemma}

\begin{proof}
Suppose $\phi \in \Deck(f^k)$. Then by Proposition~\ref{p:Sarah1}, there exists a unique $\mu \in \Deck(f)$ such that $f^{k-1} \circ \phi = \mu \circ f^{k-1}$ and  $\mu \neq \mathrm{id}$ if and only if $\phi \in \Deck^\ast(f^k)$. Furthermore, $\mu(\mathcal{V}_f) = \mathcal{V}_f$.
 
 Define $h \colon \Deck(f^k) \to \Deck(f)$ by $h(\phi) = \mu$, where $\mu$ is defined as the map from the above paragraph. We claim that $h$ is a homomorphism. To see this, note that if $f^{k-1} \circ \phi_1 = \mu_1 \circ f^{k-1}$ and $f^{k-1} \circ \phi_2 = \mu_2 \circ f^{k-1}$, then
 \begin{align*}
  f^{k-1} \circ \phi_1 \circ \phi_2  &= \mu_1 \circ f^{k-1} \circ \phi_2 \\
  &= \mu_1 \circ \mu_2 \circ f^{k-1} \\
  &= \mu_1 \circ \mu_2 \circ f^{k-1}
 \end{align*}
It follows that $h(\phi_1 \circ \phi_2) = h(\phi_1) \circ h(\phi_2)$. By Proposition~\ref{p:Sarah1}, we have $\ker h = \Deck(f^{k-1})$. Since each coset of $\ker h$ in $\Deck(f^k)$ has cardinality equal to $|\Deck(f^{k-1})|$, and since there are at most $|\Deck(f)| = d$ cosets, we conclude from the First Isomorphism Theorem that \eqref{eq:decksizes} holds.

To prove the final claim, recall that if $f$ is bicritical then all non-identity elements $\mu \in \Deck(f)$ satisfy $\mathrm{Fix}(\mu) = \mathcal{C}_f$.  Furthermore, if $\mu(\mathcal{V}_f) = \mathcal{V}_f$, then $\mu$ either fixes the elements of $\mathcal{V}_f$ pointwise, or $\mu$ is an involution which transposes the elements of $\mathcal{V}_f$. Since an involution is completely defined by its fixed points, we see that there is at most one $\mu \in \Deck(f)$ which transposes the elements of $\mathcal{V}_f$. Denote this element by $\nu$.

 Now suppose $f$ is bicritical but not a power map, so that $| \mathcal{C}_f \cup \mathcal{V}_f | \geq 3$. If $\mu \in \Deck(f)$ fixes the elements of $\mathcal{V}_f$ pointwise, then $| \mathrm{Fix}(\phi) | \geq 3$, so that $\phi$ is the identity. It follows that $\ran h \subseteq \{ \mathrm{id}, \nu \}$, and so again by the First Isomorphism Theorem we must have that the quotient \eqref{eq:decksizes} is at most $2$.
\end{proof}

To end this section, we give a result of independent interest. A general form of the following result was proved by Pakovich in \cite{Sym}, making use of algebraic curves. Here we give a dynamical proof, using properties of deck groups.

 \begin{proposition}
 Let $f$ be a bicritical rational map and $\phi \in \Deck(f^k)$. Then $f \circ \phi$ is conjugate to $f$.
 \end{proposition}
 
 \begin{proof}
 Let $\phi_k = \phi$. By Proposition~\ref{p:keyprop}, there exists $\phi_{k-1} \in \Deck(f^{k-1})$ such that $f \circ \phi_k = \phi_{k-1} \circ f$. Precomposing by $\phi_{k-1}$ gives
 \begin{equation}\label{e:conjugacyprop1}
  f \circ \phi_k \circ \phi_{k-1} = \phi_{k-1} \circ f \circ \phi_{k-1}.
 \end{equation}
 Again by Proposition~\ref{p:keyprop}, there exists $\phi_{k-2} \in \Deck(f^{k-2})$ such that $f \circ \phi_{k-1} = \phi_{k-2} \circ f$, and so \eqref{e:conjugacyprop1} becomes
 \[
 f \circ \phi_{k} \circ \phi_{k-1} = \phi_{k-1} \circ f \circ \phi_{k-1} = \phi_{k-1} \circ \phi_{k-2} \circ f.
 \]
 We can repeat the above process to recursively obtain $\phi_{j} \in \Deck(f^{j})$, so that
 \begin{equation}\label{e:conjugacyprop2}
  f \circ \phi_{k} \circ \phi_{k-1} \cdots \circ \phi_{1} = \phi_{k-1} \circ \cdots \phi_1 \circ \phi_0 \circ f.
 \end{equation}
 Furthermore, we must have $\phi_0 = \mathrm{id}$, so denoting $\Phi = \phi_{k-1} \circ \cdots \phi_{1}$, we see that \eqref{e:conjugacyprop2} becomes
 \[
  (f \circ \phi) \circ \Phi = \Phi \circ f.
 \]
 Since $\Phi$ is a M\"obius transformation, the result follows.
 \end{proof}

\section{M\"obius transformations preserving the sets of critical points and critical values values of a bicritical rational map}\label{s:specialmobius}

As can be ascertained from Proposition~\ref{p:keyprop} and Lemma~\ref{l:deckincreases}, the M\"obius transformations $\mu$ such that $\mu(\mathcal{C}_f) = \mathcal{C}_f$ and $\mu(\mathcal{V}_f) = \mathcal{V}_f$ are of particular importance when it comes to analyzing the groups $\Deck(f^k)$. In fact, when $f$ is not a power map these two conditions on $\mu$ are very restrictive.

\begin{lemma}\label{l:CfcapVfequals3}
 Let $f$ be a bicritical rational map of degree $d$ such that $|\mathcal{C}_f \cup \mathcal{V}_f| = 3$. Then the only M\"obius transformation $\mu$ satisfying $\mu(\mathcal{C}_f) = \mathcal{C}_f$ and $\mu(\mathcal{V}_f) = \mathcal{V}_f$ is the identity. Furthermore, $\Deck(f^k) \cong \Z_d$ for all $k \geq 1$
\end{lemma}

\begin{proof}
 Since $|\mathcal{C}_f \cup \mathcal{V}_f| = 3$, there exists a unique $w \in \mathcal{C}_f \cap \mathcal{V}_f$. But then any M\"obius transformation $\mu$ such that $\mu(\mathcal{C}_f) = \mathcal{C}_f$ and $\mu(\mathcal{V}_f) = \mathcal{V}_f$ must fix $w$. Therefore, $\mu$ would have to act as the identity on the three element set $\mathcal{C}_f \cup \mathcal{V}_f$, and so $\mu = \mathrm{id}$. Since $\Deck(f) \cong \Z_d$, the final claim then follows from Lemmas~\ref{l:eventuallyconstant} and \ref{l:deckincreases}.
\end{proof}

We now consider the case where $f$ is bicritical and $\mathcal{C}_f \cap \mathcal{V}_f = \varnothing$. This is equivalent to $|\mathcal{C}_f \cup \mathcal{V}_f| = 4$. First we consider the set of M\"obius transformations $\mu$ such that $\mu(\mathcal{C}_f) = \mathcal{C}_f$ and $\mu(\mathcal{V}_f) = \mathcal{V}_f$ in this case.

\begin{lemma}\label{l:fourMobius}
Let $f$ be a bicritical rational map such that $|\mathcal{C}_f \cup \mathcal{V}_f| = 4 $. Then there exist at most four M\"obius transformations $\mu$ such that $\mu(\mathcal{C}_f) = \mathcal{C}_f$ and $\mu(\mathcal{V}_f) = \mathcal{V}_f$, which are the following.
\begin{enumerate}
\item $\mu$ fixes the elements of $\mathcal{C}_f$ and $\mathcal{V}_f$ pointwise, so that $\mu = \mathrm{id}$.
\item $\mu_1$ such that $\mathrm{Fix}(\mu_1) = \mathcal{C}_f$ and $\mu_1$ swaps the elements of $\mathcal{V}_f$.
\item $\mu_2$ such that $\mathrm{Fix}(\mu_2) = \mathcal{V}_f$ and $\mu_2$ swaps the elements of $\mathcal{C}_f$.
\item $\mu_3$ such that $\mu_3$ swaps the elements of $\mathcal{C}_f$ and swaps the elements of $\mathcal{V}_f$.
\end{enumerate}
 Furthermore, each $\mu_i$, $i=1,2,3$ is an an involution. Furthermore, if all the above maps exist for the map $f$, they form a group which is isomorphic to $V_4$.
\end{lemma}

\begin{proof}
 Since a M\"obius transformation is uniquely characterized by its action on three points, we see there are the following possibilities for $\mu$. 

Furthermore, since each of the maps $\mu_1$, $\mu_2$ and $\mu_3$ have a period $2$ orbit, each one must be an involution. To complete the proof, we need to show that the four M\"obius transformations listed above form a group. But by Lemma~\ref{l:Beardon}, $\mu_1$ and $\mu_2$ commute. Furthermore, $\mu_1 \circ \mu_2$ is an involution which swaps the elements of $\mathcal{C}_f$ and swaps the elements of $\mathcal{V}_f$, so that $\mu_1 \circ \mu_2 = \mu_3$. Thus $\langle \mu_1, \mu_2 \rangle = \{ \mathrm{id}, \mu_1, \mu_2, \mu_3 \} \cong V_4$.
\end{proof}

We will continue to use the notation $\mu_i$, $i=1,2,3$ to denote the transformations obtained from the above lemma. Since any element of $\Deck(f)$ fixes $\mathcal{C}_f$ pointwise, we have $\mu_2, \mu_3 \notin \Deck(f)$.

\begin{lemma}\label{l:notpowerDeckf2}
 Let $f$ be a bicritical rational map of degree $d$, and suppose $f$ is not a power map. If $\Deck^\ast(f^2) \neq \varnothing$ then $\Deck(f^2) \cong D_{2d}$.
\end{lemma}

\begin{proof}
By Lemma~\ref{l:decksizes}, since $f$ is not a power map we have $|\Deck(f^2)| = 2|\Deck(f)| = 2d$. Since $\Deck(f) \cong \Z_d$, it follows from Theorem~\ref{t:cyclicOrDihedral} that $\Deck(f^2) \cong \Z_{2d}$ or $\Deck(f^2) \cong D_{2d}$. Furthermore, by assumption $\Deck(f^2) \neq \Deck(f)$, and so we know from Lemma~\ref{l:deckincreases} that there exists a non-identity $\mu \in \Deck(f)$ such that $\mu(\mathcal{V}_f) = \mathcal{V}_f$. By Lemma~\ref{l:CfcapVfequals3} and the fact that $f$ is not a power map, we must have $|\mathcal{C}_f \cup \mathcal{V}_f| = 4$, and so $\mu$ must be the map $\mu_1$ from Lemma~\ref{l:fourMobius}. Since $\mu_1$ swaps the elements of $\mathcal{V}_f = \{ v_1, v_2\}$ we see that $f(v_1) = f(v_2)$, and so $f$ is critically coalescing.

Now suppose that $\Deck(f^2) \cong \Z_{2d}$. By Lemma~\ref{l:localdegree}, there exists $c \in \hatC$ such that $\deg_{f^2}(c) \geq 2d$. But as $f$ is bicritical, we must have $\deg_{f^2}(c) \in \{1, d, d^2 \}$. Hence $\deg_{f^2}(c) = d^2$ and so $c \in \mathcal{C}_f \cap \mathcal{V}_f$. Therefore by Lemma~\ref{l:CfcapVfequals3}, we have $\Deck(f^2) \cong \Z_d$, which is a contradiction. Thus $\Deck(f^2) \cong D_{2d}$.
\end{proof}

The next result shows that if $f$ is bicritical but not a power map, then the group $\Deck_\infty(f) = \bigcup_{k=1}^\infty \Deck(f^k)$ studied by Pakovich is \cite{Sym} is equal to $\Deck(f^3)$.

\begin{proposition}\label{p:Deckf3isenough}
 Let $f$ be a bicritical rational map which is not a power map. Then $\Deck(f^k) = \Deck(f^3)$ for all $k \geq 3$.
\end{proposition}

\begin{proof}
If $\Deck(f) = \Deck(f^2)$ or $\Deck(f^2) = \Deck(f^3)$, then the result holds by Lemma~\ref{l:eventuallyconstant}. Thus we may assume that $\Deck(f) \subsetneq \Deck(f^2) \subsetneq \Deck(f^3)$. Since $\Deck(f) \subsetneq \Deck(f^2)$, it follows from Lemma~\ref{l:deckincreases} that, using the notation of  Lemma~\ref{l:fourMobius}, $\mu_1 \in \Deck^\ast(f)$. Similarly, since $\Deck(f^3) \neq \Deck(f^2)$, there exists $\mu \in \Deck^\ast(f^2)$ such that $\mu(\mathcal{V}_f) = \mathcal{V}_f$. Such a map must be either $\mu_2$ or $\mu_3$ from Lemma~\ref{l:fourMobius}.  But since $\mu_3 = \mu_1 \circ \mu_2$ and $\mu_2 = \mu_1 \circ \mu_3$, we see that $\mu_2 \in \Deck^\ast(f^2)$ if and only if $\mu_3 \in \Deck^\ast(f^2)$. However, this means that $\{ \mathrm{id}, \mu_1, \mu_2, \mu_3 \} \subseteq \Deck(f^2)$ and so $\Deck^\ast(f^3) \cap \{ \mathrm{id}, \mu_1, \mu_2, \mu_3 \} = \varnothing$. Thus $\Deck^\ast(f^3)$ does not contain a M\"obius transformation $\mu$ such that $\mu(\mathcal{V}_f) = \mathcal{V}_f$. But then Lemma~\ref{l:deckincreases} implies $\Deck(f^4) = \Deck(f^3)$, and so by Lemma~\ref{l:eventuallyconstant} we have $\Deck(f^k) = \Deck(f^3)$ for all $k \geq 3$.
\end{proof}

\section{Proofs of the Main Theorems}\label{s:proofs}

We are now ready to prove our main theorems.

\subsection{Proof of Theorem~\ref{t:odddegree}}

\begin{proof}[Proof of Theorem~\ref{t:odddegree}]
It is clear that if $f$ is a power map then $\Deck(f^k) \cong \Z_{d^k}$. Now suppose $f$ is not a power map, so that $|\mathcal{C}_f \cup \mathcal{V}_f| > 2$. If $|\mathcal{C}_f \cup \mathcal{V}_f| = 3$, then Lemma~\ref{l:CfcapVfequals3} asserts that $\Deck(f^k) \cong \Z_d$ for all $k$. If $|\mathcal{C}_f \cup \mathcal{V}_f| =4$ then we note that since $d$ is odd, $\Deck(f) \cong \Z_d$ cannot contain an element of order $2$. But this means none of the elements $\mu_i$, $i=1,2,3$ from Lemma~\ref{l:fourMobius} can belong to $\Deck(f)$, and so by Lemma~\ref{l:deckincreases} we have $\Deck(f^2) = \Deck(f) \cong \Z_d$. Thus by Lemma~\ref{l:eventuallyconstant}, we have $\Deck(f^k) = \Deck(f) \cong \Z_d$ for all $k \geq 1$.
\end{proof}

\subsection{Proof of Theorem~\ref{mthm}}

\begin{proof}[Proof of Theorem~\ref{mthm}]
 Again, we note that if $f$ is a power map, then $\Deck(f^k) \cong \Z_{d^k}$ for all $k$. If $f$ is not a power map, then by Lemmas~\ref{l:eventuallyconstant} and \ref{l:notpowerDeckf2}, then either $\Deck(f^2) \cong D_{2d}$ or $\Deck(f^k) \cong \Z_{d}$ for all $k \geq 1$.
 
 If $\Deck(f^2) \cong D_{2d}$, then by Theorem~\ref{t:cyclicOrDihedral} and Lemma~\ref{l:decksizes}, the only possibilities for $\Deck(f^3)$ (up to isomorphism) are $D_{4d}$ or $D_{2d}$. But by Proposition~\ref{p:Deckf3isenough}, the group $\Deck(f^k)$ cannot be larger than $\Deck(f^3)$, and this completes the proof.
\end{proof}

We conclude by showing that $D_{2d}$ and $D_{4d}$ are actually realized as $\Deck(f^k)$ for some bicritical rational map $f$ of even degree $d$. 

\begin{proposition}\label{p:examples}
 Let $d \geq 2$ be even.
 \begin{enumerate}
  \item If $f(z) = \frac{z^d-a}{z^d+a}$ for some $a \neq 0$ then $\Deck(f^2) \cong D_{2d}$.
  \item If $g(z) = \frac{z^d-1}{z^d+1}$ then $\Deck(g^3) \cong D_{4d}$. 
 \end{enumerate}
\end{proposition}

\begin{proof}
As with Example~\ref{ex:critcoalconversefalse}, one could compute the groups $\Deck(f^2)$ and $\Deck(g^3)$ by hand, but we instead use some of our previously obtained results. We first note that since $\mathcal{C}_f = \mathcal{C}_g = \{ 0, \infty \}$ and $d$ is even, the involution $\mu(z) = -z$ belongs to $\Deck(f)$ and $\Deck(g)$.
 \begin{enumerate}
  \item It is clear that $\mathcal{V}_f = \{-1,1\}$, and since $\mu(\mathcal{V}_f) = \mathcal{V}_f$, it follows from Lemmas~\ref{l:deckincreases} and \ref{l:notpowerDeckf2} that $\Deck(f^2) \cong D_{2d}$.
  \item Let $\phi(z) = \frac{1}{z}$. Then a simple calculation yields
  \[
   g \circ \phi(z) = -\frac{z^d-1}{z^d + 1} = \mu \circ g(z).
  \]
Since from the first part we know $\mu \in \Deck(g)$, we see that by Lemma~\ref{l:keypropconverse} we must have $\phi \in \Deck(g^2)$. Furthermore, it is clear that $\phi(\mathcal{V}_g) = \mathcal{V}_g = \{ -1,1\}$ and so $\Deck^\ast(g^3) \neq \varnothing$ by Lemma~\ref{l:deckincreases}. But then $|\Deck(g^3)| = 2|\Deck(g^2)| $ by Lemma~\ref{l:decksizes} and so $\Deck(g^3) \cong D_{4d}$. \qedhere
 \end{enumerate}
\end{proof}

\subsection*{Acknowledgments} S. Koch was partially supported by NSF grant \#2104649.  K. Lindsey was partially supported by NSF grant \#1901247. There are no conflicts of interest associated with this funding.
\bibliography{Deckgroups}
\bibliographystyle{plain}

\end{document}